\documentclass[reqno,a4paper,11pt]{article}
\usepackage{amsmath,amsthm,dsfont,amsfonts,amssymb,fancyhdr}
\usepackage{enumerate}
\usepackage[usenames,dvipsnames]{color}
\usepackage{bbm,soul,supertabular,longtable,verbatim,extarrows}
\usepackage{titlesec}
\usepackage{graphicx}
\usepackage{BOONDOX-cal}
\usepackage{cite}
\usepackage{tikz}
\usepackage{booktabs,multirow,makecell}
\usepackage{authblk}
\usetikzlibrary[trees]

\usepackage{lipsum}
\usepackage{mathrsfs}
\usepackage{indentfirst}
\usepackage[colorlinks=true, allcolors=blue]{hyperref}
\usepackage[top=2.4cm,bottom=2.2cm,left=2.6cm,right=2cm]{geometry}
\usepackage{cancel}

\newtheorem{theorem}{Theorem}[section]
\newtheorem{proposition}[theorem]{Proposition}
\newtheorem{lemma}[theorem]{Lemma}

\newtheorem{conjecture}[theorem]{Conjecture}
\newtheorem{claim}[theorem]{Claim}
\theoremstyle{definition}

\soulregister\cite7
\soulregister\citep7
\soulregister\citet7
\soulregister\ref7
\soulregister\pageref7

\begin{document}
\title{\textbf{There does not exist a strongly regular graph with parameters $(1911,270,105,27)$}}
\author[a,b]{Jack H. Koolen}
\author[a]{Brhane Gebremichel\footnote{Corresponding author.}}
\affil[a]{\footnotesize{School of Mathematical Sciences, University of Science and Technology of China, 96 Jinzhai Road, Hefei, 230026, Anhui, PR China.}}
\affil[b]{\footnotesize{CAS Wu Wen-Tsun Key Laboratory of Mathematics, University of Science and Technology of China, 96 Jinzhai Road, Hefei, Anhui, 230026, PR China}}
\date{}
\maketitle
\newcommand\blfootnote[1]{%
\begingroup
\renewcommand\thefootnote{}\footnote{#1}%
\addtocounter{footnote}{-1}%
\endgroup}
\blfootnote{E-mail addresses: {\tt koolen@ustc.edu.cn} (J.H. Koolen), {\tt brhaneg220@mail.ustc.edu.cn} (B. Gebremichel).}

\begin{abstract}
In this paper we show that there does not exist a strongly regular graph with parameters $(1911,270,105,27)$.
\end{abstract}

\textbf{Keywords} : strongly regular graph, join of graphs.

\textbf{Mathematics Subject Classification:} 05C50, 05E30.

\section{Introduction}
In this paper we show the following result:
\begin{theorem}\label{SRG(1911,270,105,27)}
There does not exist a strongly regular graph with parameters $(1911,270,105,27)$.
\end{theorem}
This was the largest open case of a set of feasible parameters of a strongly regular graph with smallest eigenvalue $-3$. We conjecture:

\begin{conjecture}
Let $G$ be a primitive strongly regular graph with parameters $(n,k, \lambda, \mu)$ and smallest eigenvalue $-3$. Then either $\mu \in \{6,9\}$ or $n \leqslant 276$.
\end{conjecture}

On this moment, there are twelve cases of parameter sets of putative primitive strongly regular graphs with smallest eigenvalue $-3$, $n > 276$ and $\mu \not\in \{6,9\}$ which are open. They are in Table \ref{table1} below (cf. \cite{QP17}).

To show the main result we find large cliques that intersect in many vertices. We are sure that the method we use in this paper can be generalized to the larger class of amply regular graphs (for a definition, see \cite{GKP2021}).

\begin{table}[h]\label{table1}
  \centering
  \begin{tabular}{c|c||c|c}
  \hline
  $(v,k,\lambda, \mu)$ & $\theta_0, [\theta_1]^{m(\theta_1)}, [\theta_2]^{m(\theta_2)}$ & $(v,k,\lambda, \mu)$ & $\theta_0, [\theta_1]^{m(\theta_1)}, [\theta_2]^{m(\theta_2)}$\\  \hline
  $(288,105,52,30)$ & $105, [25]^{27}, [-3]^{260}$ & $(476,133,60,28)$ & $133, [35]^{34}, [-3]^{441}$\\
  $(300,117,60,36)$ & $117, [27]^{26}, [-3]^{273}$ & $(540,147,66,30)$ & $147, [39]^{35}, [-3]^{504}$\\
  $(351,140,73,44)$ & $140, [32]^{26}, [-3]^{324}$ & $(550,162,75,36)$ & $162, [42]^{33}, [-3]^{516}$\\
  $(375,102,45,21)$ & $102, [27]^{34}, [-3]^{340}$ & $(575,112,45,16)$ & $112, [32]^{46}, [-3]^{528}$\\
  $(405,132,63,33)$ & $132, [33]^{30}, [-3]^{374}$ & $(703,182,81,35)$ & $182, [49]^{37}, [-3]^{665}$\\
  $(441,88,35,13)$ & $88, [25]^{44}, [-3]^{396}$ & $(1344,221,88,26)$ & $221, [65]^{56}, [-3]^{1287}$ \\
  \hline
\end{tabular}
  \caption{List of putative primitive strongly regular graphs with smallest eigenvalue $-3$ for $n >276$, unknown whether they exist or not.}
\end{table}

 This paper is organized as follows: In the next section we give the preliminaries. In Section \ref{sec large cliques} we give restrictions on graphs with large cliques and smallest eigenvalue at least $-3$. In Section \ref{sec on the local graph of $G$} we find large cliques in a putative strongly regular graph with parameters $(1911,270,105,27)$ and apply the restrictions given in Section \ref{sec large cliques} to show the main result.

\section{Preliminaries}\label{Pre}

\subsection{Graphs}
In this paper all the graphs are finite, undirected and simple. For definitions, we do not define, see \cite{BH11}. The eigenvalues of a graph are the eigenvalues of its adjacency matrix $A(G)$ indexed by $V(G)$, such that $A_{xy} = 1$ if $xy \in E(G)$ and $0$ otherwise. The smallest eigenvalue of a graph is denoted by $\lambda_{\min}(G)$.

Let $G=(V(G),E(G))$ be a graph. The valency $k_x$ of a vertex $x$ of $G$ is the number of neighbours of $x$, i.e. the vertices $y \in V(G)$ such that $xy \in E(G)$. A graph $G$ is $k$-regular if $k_x=k$ for all vertices $x \in V(G)$. A graph $G$ is strongly regular with parameters $(n,k,\lambda,\mu)$ if $G$ has $n$ vertices, is $k$-regular and any two distinct vertices have exactly $\lambda$ (resp. $\mu$) common neighbours if they are adjacent (resp. non-adjacent). The strongly regular graph $G$ is called primitive if $G$ and its complement are both connected.

\subsection{Interlacing}

If $M$ (resp. $N$) is a real symmetric $m\times m$ (resp. $n\times n$) matrix with $\theta_1(M)\geqslant \theta_2(M)\geqslant \cdots\geqslant \theta_m(M)$ (resp. $\theta_1(N)\geqslant\theta_2(N) \geqslant \cdots\geqslant\theta_n(N)$) the eigenvalues of $M$ (resp. $N$) in nonincreasing order. Assume $m\leqslant n$. Then we say that the eigenvalues of $M$ \emph{interlace} the eigenvalues of $N$, if $\theta_{n-m+i}(N)\leqslant\theta_i(M)\leqslant\theta_i(N)$ for $i=1,\ldots,m$.

The following result is a special case of interlacing.

\begin{lemma}[{cf. \cite[Theorem 9.1.1]{GD01}}]\label{interlacing}
Let $B$ be a real symmetric $n\times n$ matrix and $C$ be a principal submatrix of $B$ of order $m$, where $m<n$. Then the eigenvalues of $C$ interlace the eigenvalues of $B$.
\end{lemma}

As an easy consequence of Lemma \ref{interlacing}, we have the following proposition.
\begin{proposition}\label{subgraph}
Let $G$ be a graph and $H$ a proper induced subgraph of $G$. Denote by $\theta_{\min}(G)$ (resp. $\theta_{\min}(H)$) the smallest eigenvalue of $G$ (resp. $H$). Then $\theta_{\min}(G)\leqslant \theta_{\min}(H)$.
\end{proposition}

Let $G=(V,E)$ be a graph and $\pi:=\{V_1,\ldots,V_r\}$ be a partition of $V$. We say $\pi$ is an \emph{equitable partition} with respect to $G$ if the number of neighbours in $V_j$ of a vertex $u$ in $V_i$ is a constant $q_{ij}$, independent of $u$, only dependent on $i$ and $j$. For an equitable partition $\pi$ with respect to $G$, the quotient matrix $Q$ of $G$ with respect to $\pi$ is defined as $Q=(q_{ij})_{1\leqslant i,j\leqslant r}$.

\begin{lemma}[{cf.\cite[Theorem 9.3.3]{GD01}}]\label{quotientmatrix}
Let $G$ be a graph. If $\pi$ is an equitable partition of $G$ and $Q$ is the quotient matrix with respect to $\pi$ of $G$, then every eigenvalue of $Q$ is an eigenvalue of $G$.
\end{lemma}

\subsection{Terwilliger graphs}
A \emph{Terwilliger graph} is a non-complete graph such that, for any two vertices $x$ and $y$ at distance $2$, the subgraph induced by common neighbours of $x$ and $y$ forms a clique of order $c$ (for some fixed $c\geqslant 0$).

\begin{lemma}[{cf. \cite[Corollary 1.16.6 (ii)]{BCN}}]\label{terwilliger graph}
There is no strongly regular Terwilliger graph with parameters $(n,k,\lambda,\mu)$ for $k < 50(\mu-1)$.
\end{lemma}

\begin{lemma}\label{quadrangle}
If a strongly regular graph with parameters $(1911, 270, 105, 27)$ exist, then it contains an induced quadrangle.
\end{lemma}
\begin{proof}
Suppose there exist a strongly regular graph $G$ with parameters $(1911, 270, 105, 27)$ which does not contain an induced quadrangles. Then $G$ is a Terwilliger graph. As $\mu= 27$ then, by Lemma \ref{terwilliger graph}, the valency of $G$ is at least $1300$, which is a contradiction as $k =270$. This shows the lemma.
\end{proof}

\subsection{Join of graphs}
Let $G_1$ and $G_2$ be two graphs such that $V(G_1) \cap V(G_2) = \emptyset$. The \emph{join} of $G_1$ and $G_2$, denoted by $G_1 \nabla G_2$, has as vertex set $V(G_1) \cup V(G_2)$ and edge set $E(G_1) \cup E(G_2) \cup \{\{x_1,x_2\} \mid x_1 \in V(G_1), x_2 \in V(G_2)\}$. The following lemma is a consequence of Section $2.3.1$ of \cite{BH11}.

\begin{lemma}\label{join}
Let $G_i$ be a $k_i$-regular graph with $n_i$ vertices, for $i = 1,2$, such that $V(G_1) \cap V(G_2) = \emptyset$. Then the smallest eigenvalue $\lambda_{\min}(G_1 \nabla G_2)$ of the join $G_1 \nabla G_2$ satisfies
\begin{equation*}
  \lambda_{\min} = \min \{ \lambda_{\min}(G_1), \lambda_{\min}(G_2),\lambda_{\min}(Q)\}
\end{equation*}
where
\begin{gather*}
Q=\begin{pmatrix}
k_1 & n_2 \\ n_1 & k_2
\end{pmatrix}.
\end{gather*}
\end{lemma}

%
%
%

The following lemma was inspired by Cao, Koolen, Munemasa, Yoshino \cite{CKMY21}.

\begin{lemma}\label{Join with K_n}
Let $G$ be a $k$-regular graph on $n$ vertices with smallest eigenvalue $\lambda_{\min}(G) \leqslant -1$. Consider $K_t \nabla G$ for some positive integer $t$. Then $\lambda_{\min}(K_t \nabla G) =\lambda_{\min}(G)$ if and only if $(\lambda_{\min}(G)-k)(\lambda_{\min}(G) +1 -t) \geqslant nt$.
\end{lemma}

\begin{proof}
As $\lambda_{\min}(K_t) \geqslant -1$, by Lemma \ref{join} we find $\lambda_{\min}(G) = \lambda_{\min}(K_t \nabla G)$ if and only if
\begin{gather*}
\lambda_{\min} \begin{pmatrix}
t-1 & n \\ t& k
\end{pmatrix}
\geqslant \lambda_{\min}(G)
\end{gather*}
if and only if
\begin{gather*}
\det \begin{pmatrix}
t-1 -\lambda_{\min}(G)& n \\ t& k-\lambda_{\min}(G)
\end{pmatrix}
\geqslant 0, \text{ as } t \geqslant 1.
\end{gather*}
This show the lemma.
\end{proof}

\section{Large cliques}\label{sec large cliques}

Let $G$ be a strongly regular graph with parameters $(n,k,\lambda,\mu)$ and smallest eigenvalue $-m$. Let $C$ be a clique of $G$ of order $c$. Then

\begin{equation}\label{maximal clique}
  c \leqslant 1+\frac{k}{m}.
\end{equation}

The inequality (\ref{maximal clique}) is called the \emph{Delsarte bound}. Moreover, if  $c = 1+\frac{k}{m}$, then $C$ is called a \emph{Delsarte clique}. 

Let $H(a,t)$ be the graph with $1+a+t$ vertices, consisting of a complete graph $K_{a+t}$ and a vertex adjacent to exactly $a$ vertices of $K_{a+t}$.

In \cite{GKP2021}, Greaves, Koolen and Park obtained the following lemma.

\begin{lemma}\label{GKPresult1}
Let $G$ be a graph with smallest eigenvalue $\lambda =\lambda_{\min}(G)$. Assume that $G$ contains an induced $H(a,t)$. Then we have

\begin{equation}
  (a-\lambda(\lambda+1))(t-(\lambda+1)^2) \leqslant \lambda(\lambda+1))^2.
\end{equation}

\end{lemma}

In this paper we need the following consequence of Lemma \ref{GKPresult1}.

\begin{lemma}
Let $G$ be a graph with smallest eigenvalue at least $-3$. Let $C$ be a complete subgraph of $G$ with order $c$. Let $x$ be a vertex of $G$ not in $C$ with exactly $t$ neighbours in $C$. Then $t \leqslant t_{\min}$ or $t \geqslant t_{\max}$ where $t_{\min}$ and $t_{\max}$ are as in the Table $\ref{table}$.

\begin{table}[h]
  \centering
  \begin{tabular}{c|c|c||c|c|c}
  \hline\hline
  $c$ & $t_{\min}$ & $t_{\max}$ & $c$ & $t_{\min}$ & $t_{\max}$\\ \hline
  $29$ & $8$ &$23$ & $31$ & $7$ & $26$ \\ 
  $30$ & $8$ & $24$ & $32$ & $7$ & $27$\\
  \hline
\end{tabular}
  \caption{Values of $t_{\min}$ and $t_{\max}$}\label{table}
\end{table}

\end{lemma}

Using Lemma, \ref{GKPresult1} Greaves et al. \cite{GKP2021} derived a method restricting the order of maximal cliques in a strongly regular graph.

\begin{lemma}[{cf. \cite[Lemma 3.7]{GKP2021}}] \label{cubic polynomial}
Let $G$ be a strongly regular graph with parameters $(v,k,\lambda, \mu)$ having smallest eigenvalue $-m$. Let $C$ be a maximal clique of $G$ of order $c$. If $\mu > m(m-1)$ and $c > \frac{{\mu}^2}{\mu-m(m-1)} -m+1$, then
\begin{equation}\label{cubic}
  ((c+m-3)(k-c+1)-2(c-1)(\lambda -c+2))^2 - (k-c+1)^2(c+m-1)(c-(m-1)(4m-1)) \geqslant 0.
\end{equation}
\end{lemma}

We denote the polynomial on the left hand side of the inequality (\ref{cubic}) by $M_{G}(c)$. 

\begin{lemma}\label{clique}
If a strongly regular graph $G$ with parameters $(1911, 270, 105, 27)$ exists, then any clique in $G$ has order at most $32$.
\end{lemma}
\begin{proof}
Let $G$ be a strongly regular graph with parameters $(1911, 270, 105, 27)$. Then, it has smallest eigenvalue $-3$. Let $C$ be a maximal clique in $G$ of order $c$. If $c > \frac{27^2}{27-6} -3+1 = 32\frac{5}{7}$, then by Lemma \ref{clique}, we have
\begin{equation*}
  M_{G}(c) = 672c^3 - 80784c^2 +1468512c + 3277200 \geq 0
\end{equation*}
as $\mu = 27 > 6$. It is easily checked that $M_{G}(26) < 0$ and $M_{G}(97) < 0$. This means that $c \geqslant 98$. This gives a contradiction, as the Delsarte bound gives $c \leqslant 1+\frac{k}{m} = 1 + \frac{270}{3} =91$.  So we obtain that any clique has order at most $32$.
\end{proof}

\begin{lemma}\label{27 clique} Let $G$ be strongly regular graph with parameters $(1911, 270, 105, 27)$. Assume there are two cliques $C_1$ and $C_2$ such that $V(C_1) \neq V(C_2)$, each of order at least $29$, intersecting in at least $22$ vertices. Then one of the following holds:
\begin{enumerate}
  \item[(1)] There is at least one vertex $z$ in the symmetric difference $V(C_1)\triangledown V(C_2)$ which is adjacent to all vertices in $V(C_1)\triangledown V(C_2)\backslash \{z\}$,
  \item[(2)] $C_1$ and $C_2$ intersect in exactly $27$ vertices and both are maximal with order $29$.
\end{enumerate}
\end{lemma}

\begin{proof}

Let $H$ be the induced subgraph on $V(C_1)\triangledown V(C_2)$. If $H$ is complete, then we are in Case (1). So we may assume $H$ is not complete. This means that $|V(C_1)\cap V(C_2)|=: t \in \{22, 23, \ldots, 27\}$ as $\mu = 27$.

Assume $t=22$. Let $C'_1$ (resp. $C'_2$) be a sub clique of $C_1$ (resp. $C_2$) such that $V(C'_1) \supseteq V(C_1)\cap V(C_2)$, $V(C'_2) \supseteq V(C_1)\cap V(C_2)$ and $|V(C'_1)|=|V(C'_2)|=29$. Let $K$ be the induced subgraph on $V(C'_1)\cup V(C'_2)$. By Proposition \ref{interlacing} we see that $K$ has smallest eigenvalue at least $-3$. Let $\pi =\{V(C'_1) \cap V(C'_2), V(C'_1)\triangledown V(C'_2)\}$ of $V(C'_1)\cup V(C'_2)$ be a partition of $K$ with quotient matrix
\begin{gather*}
  Q= \begin{pmatrix}
21 & 14 \\ 22 & \alpha +6
\end{pmatrix}.
\end{gather*}

By Lemma \ref{quotientmatrix}, we see that the smallest eigenvalue of $Q$ is at least $-3$. This implies that $24\alpha \geqslant 92$, as $\det(Q+3\mathbf{I})\geqslant 0$. So, $\alpha \geqslant \frac{23}{6}$. This means that there are at least $\lceil \frac{7 \times 23}{6}\rceil =27$ edges  between $V(C'_1) \backslash V(C'_2)$ and $V(C'_2) \backslash V(C'_1)$. If Case (1) of the lemma does not happen, then all vertices $V(C'_1)\backslash V(C'_2)$ have at most 5 neighbours in $V(C'_2)\backslash  V(C'_1)$, as $\mu=27$.

We need to consider two cases. There exist a vertex $x \in V(C'_1)\backslash V(C'_2)$ such that $x$ has $5$ neighbours in $V(C'_2)\backslash V(C'_1)$ or there is no vertex $x \in V(C'_1)\backslash V(C'_2)$  with $5$ neighbours in $V(C'_2)\backslash V(C'_1)$.

In the first case, let $y_1, y_2, \ldots, y_5$ be the $5$ neighbours of $x$ in $V(C'_2)\backslash V(C'_1)$. Then $y_1, y_2, \ldots, y_5$ have each at most $5$ neighbours in $V(C'_1)\backslash V(C'_2)$ and hence there is an edge $zu$ between $V(C'_1)\backslash V(C'_2)$ and $V(C'_2)\backslash V(C'_1)$, such that $z \in (V(C'_1)\backslash V(C'_2))\backslash \{x\}$ and $u \in (V(C'_2)\backslash V(C'_1))\backslash \{y_1, y_2, \ldots, y_5\}$. Then $u$ and $x$ are at distance two and they have at least $28$ common neighbours, a contradiction. Now assume that all vertices of $V(C'_1)\backslash V(C'_2)$ (resp. $V(C'_2)\backslash  V(C'_1)$) have at most 4 neighbours in $V(C'_2)\backslash  V(C'_1)$ (resp. $V(C'_1)\backslash V(C'_2)$). As there are at least $27$ edges  between $V(C'_1) \backslash V(C'_2)$ and $V(C'_2) \backslash V(C'_1)$, then there exist $x \in V(C'_1)\backslash V(C'_2)$ and $y \in V(C'_2)\backslash V(C'_1)$ such that $x$ and $y$ are at distance two and they have at least $8+22 = 30$ common neighbours, a contradiction. This shows that, if $t=22$, then we are in Case (1) of the lemma.

In similar fashion, it can be shown that, if $t \in \{23, 23, \ldots, 26\}$, then we are in Case (1) of the lemma.

Now assume $t=27$. If $|V(C_1)|  \geqslant 30$ and $|V(C_2)|  \geqslant 29$, then the quotient matrix $Q$ of $\pi = \{V(C_1) \cap V(C_2), V(C_1) \backslash (V(C_2), V(C_2) \backslash (V(C_1)\}$ satisfies
 \begin{gather*}
 \begin{pmatrix}
26 & t_1 & t_2 \\ 27 & t_1-1 & 0 \\ 27 & 0 & t_2-1
\end{pmatrix},
\text{ where } t_1 +27 =|V(C_1)| \text{ and } t_2+27 = |V(C_2)|,
\end{gather*}
or we are in Case (1) of the lemma.

As the smallest eigenvalue of $Q$ is at least $-3$ we obtain that
\begin{equation*}
  29(t_1+2)(t_2+2)-27(t_1(t_2+2)+ t_2(t_1+2)) \geqslant 0
\end{equation*}
This means
\begin{equation*}
  -25t_1t_2 + 4(t_1+t_2)+ 116 \geqslant 0,
\end{equation*}
and hence
\begin{equation*}
  25(t_1 - \frac{4}{25})(t_2 - \frac{4}{25}) < 117
\end{equation*}
But, as $t_1 \geqslant 3$ and $t_2 \geqslant 2$ we have $25(t_1 - \frac{4}{25})(t_2 - \frac{4}{25}) > 130$, a contradiction. This shows the lemma.
\end{proof}

\section{On the local graph of $G$}\label{sec on the local graph of $G$}

For a graph $G$ and $x \in V(G)$ let $\Delta(x)$ be the induced subgraph on $\{y \in V(G) \mid x \sim y\}$. The graph $\Delta(x)$ is called the local graph of $G$ with respect to $x$.

\begin{lemma} [{cf. \cite{Gavrilyuk10,KP10}}]\label{coclique} 
Let $G$ be a primitive strongly regular graph with parameters $(v,k,\lambda, \mu)$. Let $x$ be a vertex of $G$ and $\Delta(x)$ be the local graph of $G$ with respect to $x$. Let $\bar{C} = \{y_1,y_2, \ldots, y_c\}$ be an independent set of $\Delta(x)$ of order $\bar{c}$. Then
\begin{equation}\label{order of co-clique}
  \binom{\bar{c}}{2}(\mu-1) \geqslant \bar{c}(\lambda+1)-k
\end{equation}
 holds. 
\end{lemma}

For distinct non-adjacent vertices $w_1,w_2 \in \Delta(x)$ define $C(w_1,w_2) := \{z \sim x \mid z \sim w_1, z \sim w_2\}$ and $c(w_1,w_2) := \text{number of elements of } C(w_1,w_2)$.

\begin{lemma}
Assume a strongly regular graph $G$ with parameters $(1911, 270, 105, 27)$ exists such that $G$ has an induced quadrangle, say $x\sim u\sim y \sim v \sim x$. Then there is no independent set $S$ of order $5$ inside $\Delta(x)$ such that $u,v \in S$.
\end{lemma}
\begin{proof}
Assume that $\Delta(x)$ contains an independent set $S=\{u_1,u_2, \ldots, u_5\}$ of order $5$. Then by Lemma \ref{coclique} we have $260 =\binom{5}{2}26= 5\times106 -270 =260$. So we have equality in (\ref{order of co-clique}). This means that $c(u_i,u_j) =26$ in $\Delta(x)$ for all $1 \leqslant i < j \leqslant 5$. As $C(u,v) \leqslant 25$ in $\Delta(x)$, we obtain that $u$ and $v$ are not both elements in an independent set $S$ of order $5$ in $\Delta(x)$. This shows the lemma
\end{proof}

\begin{lemma}\label{number of common neighbours in W}
Assume a strongly regular graph $G$ with parameters $(1911, 270, 105, 27)$ exists such that $G$ has an induced quadrangle, say $x\sim u\sim y \sim v \sim x$. Let $U= \{u,v, w_1,w_2\}$ be an independent set of $\Delta(x)$. Let $A_i =\{a_{2i-1}, a_{21}\}$ for $i = 1,2, \ldots, 6$ such that $A_i \in  \binom{U}{2}$, $A_i \neq A_j$ if $1 \leqslant i < j \leqslant 6$ and $A_1 =\{u,v\}$.  Then $c(u,v) \in \{24,25\}$ and $\sum_{\{u_1,u_2\} \in \binom{U}{2}} c(u_1,u_2) \in \{154, 155\}$. Then one of the following hold:

\begin{enumerate}
   \item[(1)] $c(u,v) \in \{24,25\}$ and $\sum_{\{u_1,u_2\} \in \binom{U}{2}} c(u_1,u_2) =154$. Then, without loss of generality,
   \begin{equation*}
     (c(a_1,a_2),c(a_3,a_4), \ldots, c(a_{11},a_{12})) \in \{(24, 26, 26, \ldots, 26), (25, 25, 26, \ldots, 26)\}.
   \end{equation*}
    Moreover, any vertex $w$ of $\Delta(x)$ has at most $2$ neighbours in $U$.
   \item[(2)] $c(u,v) = 25$  and $\sum_{\{u_1,u_2\} \in \binom{U}{2}} c(u_1,u_2) = 155$. Then
   \begin{equation*}
     (c(a_1,a_2),c(a_3,a_4), \ldots, c(a_{11},a_{12})) = (25, 26, 26, \ldots, 26).
   \end{equation*}
  In this case there is a unique vertex $z$ of $\Delta(x)$ with exactly $3$ neighbours in $U$ and any other vertex $w$ of $\Delta(x)$ has at most $2$ neighbours in $U$.
\end{enumerate}
\end{lemma}

\begin{proof}
As $k= 270$ and $\lambda = 105$, we have $4(\lambda +1)-k = 154$. This means that the number of vertices $w \in \Delta(x)$ such that $w$ is adjacent to at least two vertices in $U$ is $154$, by Lemma \ref{coclique}. This implies \begin{equation*}
  \sum_{\{u_1,u_2\} \in \binom{U}{2}} c(u_1,u_2)  \geqslant 154.
\end{equation*}

As $c(u,v) \leqslant 25$ and $c(u_1,u_2) \leqslant 26$ for all $\{u_1,u_2\} \in \binom{U}{2}$ we find that
\begin{equation*}
  154 \leqslant \sum_{\{u_1,u_2\} \in \binom{U}{2}} c(u_1,u_2) \leqslant 155.
\end{equation*}
We also find that $c(u,v) \geq 154 - 5 \times 26 = 24$.

If $c(u,v) =24$, then $c(a_{2i-1},a_{2i})=26$ for $i = 2,3, \ldots, 6$ and any vertex $w$ in $\Delta(x)$ has at most two neighbours in $U$.

If $c(u,v) = 25$  and $\sum_{\{u_1,u_2\} \in \binom{U}{2}} c(u_1,u_2) = 154$, then there exists at most one $i$  in $\{2,3, \ldots, 6\}$ such that $c(a_{2i-1},a_{2i}) = 25$ and for the other $i$'s in $\{2,3, \ldots, 6\}$ we have $c(a_{2i-1},a_{2i}) = 26$. 

If $c(u,v) = 25$  and $\sum_{\{u_1,u_2\} \in \binom{U}{2}} c(u_1,u_2) = 155$, then $c(a_{2i-1},a_{2i})=26$ for all $i = 2,3, \ldots, 6$ and there exists a unique vertex $z$ of $\Delta(x)$ such that $z$ has exactly $3$ neighbours in $U$. Any other vertex $w$ of $\Delta(x)$ has at most $2$ neighbours in $U$. This shows the lemma.
\end{proof}

As a consequence of Lemma \ref{number of common neighbours in W} we have the following.

\begin{lemma}\label{the graph on W and its valency}
Assume a strongly regular graph $G$ exists with parameters $(1911, 270, 105, 27)$ such that $G$ has an induced quadrangle, say $x \sim u \sim y \sim v \sim x$. Let $W: = \{ w\sim x \mid w\neq u, w\neq v, w \not\sim u, w\not\sim v\}$ and let $Z:= \{z \in W \mid \text{there exists a vertex } z' \in W\backslash \{z\} \text{ such that } z' \not\sim z\}$. Let $\Gamma_{W}$ be the subgraph induced on $W$. For $w \in W$, let $k_w$ be the valency of $w$ in $\Gamma_{W}$. For $z \in Z$, let $K_z$ (resp. $\tilde{K}_z$) be the subgraph induced on $\{ w \in W \mid w \neq z, w \not\sim z\}$ (resp. $\{w \in W \mid w \neq z, w \not\sim z\} \cup \{x\}$). Then the following hold:
\begin{enumerate}
  \item[(1)] The graph $K_z$ and $\tilde{K}_z$ are complete and $K_z$ has at least $28$ vertices,
  \item[(2)] $|W| \in \{82,83\}$,
  \item[(3)] For $w \in W$ we have $k_w \in \{53, 54, |W|-1\}$. Moreover, if $|W| =82$, then $k_w \in \{53, |W|-1\}$ and any two distinct non-adjacent vertices $z,z' \in Z$ have $c(z,z') =26$.
\end{enumerate}

\end{lemma}

\begin{proof}
(1): If $K_z$ is not complete, then there would be an independent set $U$ of order $5$ containing $u$ and $v$, a contradiction with Lemma \ref{coclique}. As $K_z$ is a subgraph of $\Delta(x)$, it is clear that $\tilde{K}_z$ is complete as well. We will show later that $K_z$ has at least $28$ vertices.

(2): We have $|W| = k-2(\lambda +1) + c(u,v)$. As $k=270$, $\lambda = 105$ and $c(u,v) \in \{24,25\}$, we find $|W| \in \{82,83\}$. This shows (2).

(3): Let $w \in W$. If $w \not\in Z$, then $k_w =|W|-1$. So, let $w \in Z$ and $w' \in Z\backslash \{z\}$ such that $w \not\sim w'$. We have $c(u,v) \in \{24,25\}$.

If $c(u,v)=24$, then $|W|=82$, and $c(a_1,a_2) =26$ for all $\{a_1,a_2\} \in \binom{\{u,v,w,w'\}}{2}$ and $\{a_1,a_2\} \neq \{u,v\}$, by Lemma \ref{number of common neighbours in W}(1). In particular, we have $c(w,w')=26$ and $k_w =k_{w'} = 105-2\times26 = 53$. In this case, $K_w$ has $82-53-1=28$ vertices.

If $c(u,v)=25$, then $|W|=83$. We have $c(u,w), c(v,w) \in \{25,26\}$. If one of $c(u,w)$ or $c(v,w)$ is equal to $25$, then $k_w = 105-c(u,w) -c(v,w)= 105-25-26=54$, by Lemma \ref{number of common neighbours in W}(1). If $c(u,w) =c(v,w) =26$, then $54 = 105-c(u,w)-c(v,w) +1 \geqslant 105-c(u,w)-c(v,w) =53$, by Lemma \ref{the graph on W and its valency}(2). This implies $k_w$ has at least $83-54-1 = 28$ vertices. This shows the lemma.
\end{proof}

\begin{lemma}\label{on 83 vertices}
Assume a strongly regular graph $G$ with parameters $(1911,270,105,27)$ exists containing an induced quadrangle, say $x \sim u \sim y \sim v \sim x$. Let $W: = \{ w\sim x \mid w\neq u, w\neq v, w \not\sim u, w\not\sim v\}$. Then $|W| \neq  83$.
\end{lemma}

\begin{proof}
Let $\Gamma_{W}$ be the subgraph of $G$ induced on $W$. In Lemma \ref{the graph on W and its valency}(3), we have seen that the valency $k_w$ in $\Gamma_W$ of a vertex $w \in W$ satisfies $k_w \in \{82,54,53\}$.

Let $Y: = \{w \in W \mid k_w \in \{54,82\}\} \cup \{x\}$.

\begin{claim}\label{complete subgraph}
The induced subgraph of $\Gamma_W$ on $Y$ is complete.
\end{claim}

\noindent{\bf Proof of Claim \ref{complete subgraph}.} Clearly $x$ is adjacent to all the other vertices in $Y$. Let $w \in W$ such that $k_w \in \{82,54\}$. If $k_w =82$, then $w$ is adjacent to all other vertices of $\Gamma_W$. So, we only need to show that, if $w$ and $w'$ are distinct vertices in $W$ such that $k_w = k_{w'} = 54$, then $w \sim w'$.

Let $w$ be such that $k_w =54$. Let $w' \in W\backslash \{w\}$ be such that $w \not\sim w'$. There are, without loss of generality, two cases namely, $c(u,w) =25$ and $c(v,w)=26$ or, $c(u,w)=c(v,w) =26$. In the first case $c(u,w') = 26= c(w',v)$ and $k_{w'} = 105-c(u,w')-(v,w') = 53$, by Lemma \ref{number of common neighbours in W}(1). In the second case there exists a vertex $z \sim x$ such that $u \sim z \sim v$ and $z \sim w$. This means that $c(u,w') = 26 = c(v,w')$ and $k_{w'} = 105-c(u,w')-c(v,w') =53$, by Lemma \ref{number of common neighbours in W}(2). \qed

By Lemma \ref{clique}, a clique of $G$ has order at most $32$, so $|Y| \leqslant 32$. This means that $|\{w \in W \mid k_w =53\}| \geqslant 83-31 = 52$. So there are two distinct non-adjacent vertices $z$ and $z'$ in $W$ with $k_z = k_{z'} =53$. We have $c(z,z') \in \{25,26\}$. There is no vertex $\hat{z}$ that is adjacent to $u,v$ and one of $z$ and $z'$. This means that, if $c(z,z') =26$, then there exists a vertex $\tilde{z}$ adjacent to $z,z'$ and one of $u$ and $v$. This means $|C(z,z') \cap W|=25$ holds whether $c(z,z') = 25$ or $c(z,z') = 26$. This means that the cliques $K_z$ and $K_{z'}$ both have exactly $29$ vertices, where $K_z$ is defined as in Lemma \ref{the graph on W and its valency}.

Assume that there are two distinct vertices in $C(z,z')$ that are adjacent to all vertices in $K_z$. Then they have at least $28$ common neighbours and hence are adjacent. This implies there are at most $2$ vertices in $C(z,z')$ such that they are adjacent to all vertices in $\tilde{K}_z$ and at most $2$ vertices in $C(z,z')$ such that they are adjacent to all vertices in $\tilde{K}_{z'}$, where $K_z$ and $\tilde{K}_z$ are as defined in Lemma \ref{the graph on W and its valency}. So there exists a vertex $w \in C(z,z')$ that is not adjacent to all vertices of $\tilde{K}_z$ and not adjacent to all vertices of $\tilde{K}_{z'}$. We find $k_w \in \{53,54\}$. Either $w$ is adjacent to all other vertices in $C(z,z')$ or there exist $w' \in C(z,z')\backslash\{w\}$ such that $w \not\sim w'$.

\begin{claim}\label{intersection of 3 vertices in W}
$|G(w) \cap G(x) \cap G(z) \cap G(z')|\geqslant 22$.
\end{claim}

\noindent{\bf Proof of Claim \ref{intersection of 3 vertices in W}.} We may assume that there exists a vertex $w' \in C(z,z')\backslash \{w\}$ such that $w \not\sim w'$, as otherwise we are done. We have $|G(x)\cap G(z)\cap G(z')| = c(z,z') = 25$, $|G(z)\cap G(z')| =27$ and $|G(w)\cap G(z)\cap G(z')| \geqslant 24$. The last statement follows from Lemma \ref{number of common neighbours in W}, as $w \sim z \sim w' \sim z' \sim w$ is an induced quadrangle of $G$. Now $|G(w) \cap G(x) \cap G(z) \cap G(z')|\geqslant 25+24-27 = 22$. This shows Claim \ref{intersection of 3 vertices in W}. \qed

Claim \ref{intersection of 3 vertices in W} implies that $w$ has at least $21$ neighbours in $W \cap C(z,z')$. So, without loss of generality, the vertex $w$ has at least $ \frac{53-25}{2}=14$ neighbours in $K_z$ and hence at least $15$ vertices in $\tilde{K}_z$. By Table \ref{table}, we find that $w$ has at least $24$ neighbours in $\tilde{K}_z$ and hence at least $23$ neighbours in $K_z$. This means that $w$ has at most $54-21-23= 10$ neighbours in $K_{z'}$, by Claim \ref{intersection of 3 vertices in W}. By Table \ref{table}, we find that $w$ has at most $8$ neighbours in $\tilde{K}_{z'}$.

Now we consider $K_w$. Then $K_w$ has at least $28$ vertices, and $\tilde{K}_w$ and $\tilde{K}_{z'}$ intersect in at least $30-8=22$ vertices.

Now consider a maximal clique $C_1$ of $G$ containing $\tilde{K}_w$ and a maximal clique $C_2$ of $G$ containing $\tilde{K}_{z'}$. We have $w \sim z \in V(K_{z'})$. This means that $C_2$ does not contain $C_1$. Also $K_w \cap K_z \neq \emptyset$, so $C_1$ does not contain $C_2$. This means that $C_1 \neq C_2$. As $|V(C_1)| \geqslant 29$ and $|V(C_2)| \geqslant 30$, by Lemma \ref{27 clique}, we find that $C_1$ or $C_2$ is not maximal, a contradiction. This shows the lemma.

\end{proof}

\begin{lemma}\label{on 82 vertices}
Assume a strongly regular graph $G$ with parameters $(1911,270,105,27)$ exists containing an induced quadrangle, say $x \sim u \sim y \sim v \sim x$. Let $W: = \{ w\sim x \mid w\neq u, w\neq v, w \not\sim u, w\not\sim v\}$. Then $|W| \neq  82$.
\end{lemma}

\begin{proof}
Let $\Gamma_{W}$ be the subgraph of $G$ induced on $W$. In Lemma \ref{the graph on W and its valency}(3), we have seen that the valency $k_w$ in $\Gamma_W$ of a vertex $w \in W$ satisfies $k_w \in \{82, 53\}$.  There are two distinct vertices $z,z' \in W$ such that $z \not\sim z'$ and hence $k_z = k_{z'} =53$, as $\Gamma_W$ is not complete.

If there exists a vertex $w' \in W$ such that $k_{w'} =81$, then consider the induced subgraphs $\hat{K}_z$ (resp. $\hat{K}_{z'}$) on $\{w \in W \mid w \neq z, w \not\sim z\} \cup \{w'\} \cup \{x\}$ (resp. $\{w \in W \mid w \neq z', w \not\sim z'\} \cup \{w'\} \cup \{x\}$). As $u$ and $v$ does not lie in an independent set of order $5$ inside $\Delta(x)$, we see that $\hat{K}_z$ and $\hat{K}_{z'}$ are complete. Now the proof follows the proof of Lemma \ref{on 83 vertices} in this case, as $|V(\hat{K}_z)|=30=|V(K)_{z'}|$. We leave the details for the reader.

So we may assume that all vertices of $\Gamma_W$ have valency $53$ and any two distinct non-adjacent vertices in $\Gamma_W$ have exactly $26$ common neighbours. Note that, by Lemma \ref{clique}, any clique $C$ in $\Gamma_W$ has at most $31$ vertices, as the induced subgraph of $G$ on $\{x\} \cup V(C)$ is complete.

Consider the join $K_4\nabla \Gamma_W$. Then $\lambda_{\min}(K_4\nabla \Gamma_W) \geqslant -3$ by Lemma \ref{Join with K_n}, as $\lambda_{\min}(\Gamma_W) \geq \lambda_{\min}(G) =-3$.

Let $w$ be a vertex of $\Gamma_W$. Let $K_w$ be the subgraph on $\{z \in W \mid z \neq w, z\not\sim w\}$. As before $K_w$ is complete and has $28$ vertices. Then, in $K_4\nabla \Gamma_W$, we consider the clique $K_4\nabla K_w$. By Table \ref{table}, we find that any vertex $z$ of $K_4\nabla \Gamma_W$ outside $K_4\nabla K_w$ has at most $7$ neighbours or at least $27$ neighbours in $K_4\nabla K_w$. This means that $z$ has at most $3$ neighbours or at least $23$ neighbours in $K_w$.

Let $w,w'$ be two distinct non-adjacent vertices in $\Gamma_W$. There are at most $3$ vertices in $C(w,w')$ that are adjacent to all vertices in $K_w$, and similarly there are at most $3$ vertices in $C(w.w')$ that are adjacent to all vertices in $k_{w'}$ as $|V(K_w)|=28=|V(K_{w'})|$. So there exist a vertex $z \in C(w,w')$ that is not adjacent to some vertex $p$ (resp. $p'$) in $K_w$ (resp. $K_{w'}$). As $z$ has at least $22$ neighbours in $C(w,w')$, as in Claim \ref{intersection of 3 vertices in W} of Lemma \ref{on 83 vertices}, without loss of generality, we see that $z$ has at most $\lfloor \frac{53-22}{2} \rfloor = 15 < 23$ neighbours in $K_w$. So $z$ has at most $3$ neighbours in $K_w$. So the two cliques $K_z$ and $K_w$ intersect in at least $25$ vertices.

Now consider the induced subgraphs $\tilde{K}_z$ and $\tilde{K}_w$ on $\{x\} \cup V(K_z)$ and $\{x\} \cup V(K_w)$ respectively. Then $\tilde{K}_z$ and $\tilde{K}_w$ both have $29$ vertices and intersect in at least $26$ vertices. As $z' \in V(K_z)\backslash V(K_w)$ and $p' \in V(K_w) \backslash V(K_z)$ we see that $\tilde{K}_z$ and $\tilde{K}_w$ must intersect in precisely $27$ vertices and both are maximal, by Lemma \ref{27 clique}. We have $z' \in V(K_z)\backslash V(K_w)$ and $p' \in V(K_w) \backslash V(K_z)$ and, $K_z$ and $K_w$ intersect in exactly $26$ vertices.

Now consider $K_{z'}$. The clique $K_{z'}$ and $K_w$ intersects in exactly two vertices under which $p'$. Consider any vertex $q$ of the clique $K_{p'}$. Then
\begin{enumerate}
  \item[(1)] $q$ has $26$ neighbours in $K_{z'}$,
  \item[(2)] $q$ has $26$ neighbours in $K_w$, or
  \item[(3)] $q$ has at most $25$ neighbours in $K_{z'}$ and at most $25$ neighbours in $K_w$.
\end{enumerate}
There are at most $5$ vertices in case (1), and at most $5$ vertices in case (2). This means that there exists a vertex $q$ in the clique $K_{p'}$ which has at most $25$ neighbours in $K_{z'}$ and at most $25$ neighbours in $K_w$. Without loss of generality, we may assume that $q$ has at most $\frac{26}{2} = 13$ neighbours in $K_{z'}$. Then $q$ has at most $3$ neighbours in $K_{z'}$ and, $K_q$ and $K_{z'}$ intersect in at least $25$ vertices. This implies the induced subgraphs $\tilde{K}_q$ and $\tilde{K}_{z'}$, on $\{x\} \cup V(K_q)$ and $\{x\} \cup V(K_{z'})$ respectively, have both $29$ vertices and they intersect in $27$ vertices, by Lemma \ref{27 clique}. But this means that $K_q$ and $K_{z'}$ intersect in exactly $26$ vertices and, $K_q$ and $K_{z'}$ intersect in exactly $4$ vertices.

We obtain that every vertex in $V(K_q)\backslash V(K_w)$ has at least $23$ neighbours in $V(K_w)$. So this means that there are at least $24 \times 21$ edges between $V(K_q)\backslash V(K_w)$ and $V(K_w)\backslash V(K_q)$. 

This implies that the number of edges between $V(K_{z'}) \cup V(K_w)$ and $V(K_{p'})$ is at most $2\times26\times26 - 2\times24\times21= 344$. On the other hand every vertex in $K_{p'}$ has valency $53$, so the number of edges between $V(K_{z'})\cup V(K_w)$ and $V(K_{p'})$ is exactly $26\times28=728 > 344$, a contradiction. This shows the lemma.
\end{proof}

Now we give the proof of the main theorem:

\noindent{\bf Proof of Theorem \ref{SRG(1911,270,105,27)}.}
By Lemma \ref{quadrangle}, we see that $G$ has an induced quadrangle, say $x \sim u \sim y \sim v \sim x$.  Let $W = \{ w \sim x \mid w \neq u, w \neq v, w \not\sim u, w\not\sim v\}$. Then, by Lemma \ref{the graph on W and its valency}, we have $|W| \in \{82,83\}$. By Lemmas \ref{on 83 vertices} and \ref{on 82 vertices}, we see that this is not possible. This finishes the proof. \qed

\section*{Acknowledgments}
\indent
J.H. Koolen is partially supported by the National Natural Science Foundation of China (No. 12071454), Anhui Initiative in Quantum Information Technologies (No. AHY150000), and
the project "Analysis and Geometry on Bundles" of Ministry of Science and Technology of the People's Republic of China.

Brhane Gebremichel is supported by a Chinese Scholarship Council at University of Science and Technology of China, China.

\bibliographystyle{plain}
\bibliography{KG202108}

\begin{thebibliography}{1}

\bibitem{BCN}
A.~E. Brouwer, A.~M. Cohen, and A.~Neumaier.
\newblock {\em Distance-Regular Graphs}.
\newblock Springer-Verlag, Berlin Heidelberg, 1989.

\bibitem{BH11}
A.~E. Brouwer and W.~H. Haemers.
\newblock {\em Spectra of Graphs}.
\newblock Springer, New York, 2012.

\bibitem{CKMY21}
M.-Y. Cao, J.~H. Koolen, A.~Munemasa, and K.~Yoshino.
\newblock Maximality of {Seidel} matrices and switching roots of graphs.
\newblock {\em Graphs Combin.}, 2021.

\bibitem{Gavrilyuk10}
A.~L. Gavrilyuk.
\newblock On the {Koolen-Park} inequality and {Terwilliger} graphs.
\newblock {\em Elec. J. Combin.}, 17:{\#}R125, 2010.

\bibitem{GD01}
C.~Godsil and G.~Royle.
\newblock {\em Algebraic Graph Theory}.
\newblock Springer-Verlag, Berlin, 2001.

\bibitem{GKP2021}
G.~R. Greaves, J.~H. Koolen, and J.~Park.
\newblock Augmenting the {Delsarte} bound: a forbidden interval for the order
  of maximal cliques in strongly regular graphs.
\newblock {\em European J. Combin.}, 97:103384, 2021.

\bibitem{KP10}
J.~H. Koolen and J.~Park.
\newblock Shilla distance-regular graphs.
\newblock {\em European J. Combin.}, 31:2064--2073, 2010.

\bibitem{QP17}
Z.~Qiao and Y.~Pan.
\newblock A note on optimistic strongly regular graphs.
\newblock {\em J. University of Science and Technology of {China}},
  47(3):197--203, 2017.

\end{thebibliography}

\end{document}